\documentclass[letterpaper,10pt]{article}

\usepackage[margin=1in]{geometry}
\usepackage[utf8]{inputenc}
\usepackage{amsmath}
\usepackage{mathtools}
\usepackage{amsthm,amssymb}
\usepackage{amsfonts}
\usepackage[numbers,square]{natbib}
\usepackage{comment}
\usepackage{authblk}
\usepackage{blindtext}
\usepackage{multicol}
\usepackage{todonotes}
\usepackage{url}
\usepackage{microtype}
\usepackage{colonequals}

\numberwithin{equation}{section}

\newtheorem{lem}{Lemma}[section]

\newtheorem{prop}[lem]{Proposition}

\newtheorem{theo}[lem]{Theorem}
\newtheorem{conj}[lem]{Conjecture}

\newtheorem{cor}[lem]{Corollary}
\newtheorem{exa}[lem]{Example}
\theoremstyle{definition}
\newtheorem{defi}[lem]{Definition}
\newtheorem{rem}[lem]{Remark}

\newcommand{\leref}[1]{Lemma \ref{#1}}
\newcommand{\theref}[1]{Theorem \ref{#1}}
\newcommand{\conjref}[1]{Conjecture \ref{#1}}

\newcommand{\coref}[1]{Corollary \ref{#1}}

\newcommand{\propref}[1]{Proposition \ref{#1}}
\newcommand{\remref}[1]{Remark \ref{#1}}

\newcommand{\R}{\mathbb{R}}

\newcommand{\ddt}{\frac{\textrm{d}}{\textrm{d}t}}

\newcommand{\equi}{\Leftrightarrow}

\DeclareMathOperator{\e}{e}

\begin{document}
\title{On the sum of squared logarithms inequality and related inequalities}
\author{Fozi M. Dannan\footnote{Arab International University,  Syria, email: fmdan@scs-net.org}
	\quad and\quad Patrizio Neff\footnote{Corresponding author: Patrizio Neff, Head of Lehrstuhl für Nichtlineare Analysis und Modellierung, Fakultät für Mathematik, Universität Duisburg-Essen, Thea-Leymann Str. 9, 45127 Essen, Germany, email: patrizio.neff@uni-due.de} 
	\quad and\quad Christian Thiel\footnote{Lehrstuhl für Nichtlineare Analysis und Modellierung, Fakultät für Mathematik, Universität Duisburg-Essen, Thea-Leymann Str. 9, 45127 Essen, Germany, email: christian.thiel@uni-due.de}}
\date{\today}
\maketitle

\begin{abstract}
\noindent
We consider the sum of squared logarithms inequality and investigate possible connections with the theory of majorization. We also discuss alternative sufficient conditions on two sets of vectors $a,b\in\R_+^n$ so that
\begin{equation}
\sum_{i=1}^n(\log a_i)^2\ \leq\ \sum_{i=1}^n(\log b_i)^2\,.\notag
\end{equation}
Generalizations of some inequalities from information theory are obtained, including a generalized information inequality and a generalized log sum inequality, which states for $a,b\in\R_+^n$ and $k_1,\ldots,k_n\in [0,\infty)$:
\begin{equation}
\sum_{i=1}^na_i\,\log\prod_{s=1}^m\left(\frac{a_i}{b_i} + k_s\right)\ \geq\ \log\prod_{s=1}^m(1+k_s)\,.\notag
\end{equation}
\end{abstract}

\bigskip


\bigskip


\tableofcontents
\section{Introduction - the sum of squared logarithms inequality}
The \emph{Sum of Squared Logarithms Inequality} (SSLI) was introduced in 2013 by B\^{\i}rsan, Neff and Lankeit \cite{Neff_log_inequality13}, with the authors giving a proof for $n\in\{2,3\}$. Recently, Pompe and Neff \cite{pompe2015generalised} have shown the inequality for $n=4$, in which case it reads: Let $a_1,a_2,a_3,a_4,b_1,b_2,b_3,b_4>0$ be given positive numbers such that
\begin{align}
a_1+a_2+a_3+a_4 &\ \leq\ b_1+b_2+b_3+b_4 \,, \notag\\
a_1\, a_2+a_1\, a_3+ a_2\, a_3+a_1\, a_4+a_2\, a_4+a_3\, a_4 &\ \leq\ b_1\, b_2+b_1\, b_3+ b_2\, b_3+b_1\, b_4+b_2\, b_4+b_3\, b_4 \,,\notag \\
a_1 \, a_2\, a_3+a_1\, a_2\, a_4+a_2\, a_3\, a_4+a_1\, a_3\, a_4 &\ \leq\ b_1\, b_2\, b_3+b_1\, b_2\, b_4+b_2\, b_3\, b_4+b_1\, b_3\, b_4 \,, \notag\\
a_1 \, a_2\, a_3\, a_4&\ =\ b_1\, b_2\, b_3\, b_4\,. \notag
\end{align}
Then
\begin{align}
(\log a_1)^2+(\log a_2)^2+(\log a_3)^2+(\log a_4)^2 \ \leq\ (\log b_1)^2+(\log b_2)^2+(\log b_3)^2+(\log b_4)^2 \, .\notag
\end{align}
The general form of this inequality can be conjectured as follows. 
\begin{defi}\label{defiElementarySymmetricPolynomial}
Let $x\in\R^n$. We denote by $e_k(x)$ the $k$-th \emph{elementary symmetric polynomial}, i.e.~the sum of all $\binom nk$ products of exactly $k$ components of $x$, so that
\begin{equation}
e_k(x) \colonequals \sum_{i_1<\ldots<i_k}x_{i_1}x_{i_2}\ldots x_{i_k}\qquad\textrm{for all\ \,$k\in\{1,\ldots,n\}$}\,;\notag
\end{equation}
note that $e_n(x)=x_1\cdot x_2\ldots \cdot x_n$.
\end{defi}

\begin{conj}[Sum of squared logarithms inequality]\label{conjSumOfSquaredLogarithmsInequality}
Let $a,b\in\R_+^n$. If
\begin{alignat*}{2}
e_k(a) \leq e_k(b)\qquad\textrm{for all\ \,$k\in\{1,\ldots,n-1\}$}
\end{alignat*}
and $e_n(a) = e_n(b)$, then
\begin{equation}\label{SumOfSquaredLogarithmsInequality}
\sum_{i=1}^n(\log a_i)^2\ \leq\ \sum_{i=1}^n(\log b_i)^2\tag{SSLI}\,.
\end{equation}

Alternatively, we can express the same statement as a minimization problem:

Let $a\in\R_+^n$ be given and define
\begin{equation}
\mathcal{E}_a\ \colonequals\ \bigl\{\,b\in\R_+^n\ |\ e_k(a)\leq e_k(b)\quad \textrm{for all\ \,$k\in\{1,\ldots,n-1\}$\ \ and\ \, $e_n(a)=e_n(b)$}\bigr\}\,.\notag
\end{equation}
Then
\begin{equation}
\inf_{b\in\mathcal{E}_a}\biggr\{ \sum_{i=1}^n(\log b_i)^2\biggr\}\ =\ \sum_{i=1}^n(\log a_i)^2\,.\notag
\end{equation}
\end{conj}

The sum of squared logarithms inequality (SSLI) has important applications in matrix analysis and nonlinear elasticity theory \cite{Neff_log_inequality13,Neff_Eidel_Osterbrink_2013, neff2014exponentiated, neff2015exponentiated, grioli2013, Neff_nagatsukasa_logpolar13}.
We notice that the previous conjecture for $n\geq 2$ puts conditions on the elementary symmetric polynomials of the numbers $a_1, \ldots , a_n$ and $b_1,\ldots , b_n$ for obtaining (SSLI).

In this article we obtain (SSLI)\ and the so called sum of powered logarithms inequality  \eqref{eqSumOfPoweredLogarithms} under some alternative conditions. We also introduce some new inequalities for the exponential functions. An extension of the log sum inequality is also obtained, which yields generalizations of the information inequality.

\section{Preliminaries}
The concept of \emph{majorization} is of great importance to \eqref{SumOfSquaredLogarithmsInequality} and related inequalities. In the following we state the basic definitions as well as some fundamental properties of majorization. For a larger survey we refer to Marshall, Olkin and Arnold \cite{marshall2010inequalities}.

There are various ways to define Majorization. Due to Hardy, Littlewood and P{\'o}lya \cite{hardy1952inequalities} we can formulate the following theorem:
\begin{theo}\label{defmaj}
Let $x,y\in\R_+^n$, then the following are equivalent:
\begin{itemize}
\item[a)] $\sum_{i=1}^nx_i^\downarrow = \sum_{i=1}^ny_i^\downarrow$ and $\sum_{i=1}^kx_i^\downarrow \leq \sum_{i=1}^ky_i^\downarrow$ for all $1\leq k\leq n-1$.
\item[b)] $\sum_{i=1}^nf(x_i)\leq \sum_{i=1}^nf(y_i)$ for all convex continous functions $f\colon \R\to\R$.
\end{itemize}
If the expressions hold, $x$ is said to be majorized by $y$, written $x\prec y$.
\end{theo}
Similarly, the concept of the weak majorization can be fomulated (see Tomic \cite{tomic1949theoreme} and Weyl \cite{weyl1949inequalities}): 
\begin{theo}\label{defmajw}
Let $x,y\in\R_+^n$, then the following are equivalent:
\begin{itemize}
\item[a)] $\sum_{i=1}^kx_i^\downarrow \leq \sum_{i=1}^ky_i^\downarrow$ for all $1\leq k\leq n$.
\item[b)] $\sum_{i=1}^nf(x_i)\leq \sum_{i=1}^nf(y_i)$ for all convex monotone increasing continous functions $f\colon \R\to\R$.
\end{itemize}
If the expressions hold, $x$ is said to be weakly majorized by $y$, written $x\prec_w y$.

If $x$ is weakly majorized by $y$, we state, in particular, that it is weakly majorized from below. Analogously we can also characterize weak majorization from above ( $x \prec^w y$), were the inequality in a) holds with greater or equal for all vectors in descending order and equivalent die inequality in b) for all convex monotone decreasing functions.
\end{theo}

The following lemma, which shows elementary properties of the so-called \emph{(weak) logarithmic majorization}, follows directly from the logarithmic laws and the monotonicity of the logarithm.
\begin{lem}[Logarithmic majorization]\label{defiLogarithicMajorization}
Let $x,y\in \R_+^n$. Then
\begin{alignat}{2}
\log x \prec_w \log y&\qquad\textrm{if and only if}\qquad x_1^\downarrow\cdot\ldots\cdot x_k^\downarrow\ \leq\ y_1^\downarrow\cdot\ldots\cdot y_k^\downarrow\quad\textrm{for all\ \,$k\in\{1,\ldots,n\}$}\,,\notag \\
\log x \prec^w \log y&\qquad\textrm{if and only if}\qquad x_1^\uparrow\cdot\ldots\cdot x_k^\uparrow\ \geq\ y_1^\uparrow\cdot\ldots\cdot y_k^\uparrow\quad\textrm{for all\ \,$k\in\{1,\ldots,n\}$}\,,\notag \\
\log x \prec \log y&\qquad\textrm{if and only if}\qquad \log x \prec_w \log y\quad \textrm{and}\quad x_1\cdot\ldots\cdot x_n = y_1\cdot\ldots\cdot y_n\,,\notag
\end{alignat}
where we abbreviate $\log z \colonequals (\log z_1, \log z_2, \ldots, \log z_n)$ for $z\in\R_+^n$.
\end{lem}

\begin{prop}\label{propImplicationsForLogMajorization}
Let $x,y \in\R_+^n$. Then
\begin{alignat}{5}
&&\log x &\prec_w \log y\quad&&\textrm{implies}&\quad x &\prec_w y&&\notag\\
&&x &\prec^w y\quad&&\textrm{implies}&\quad \log x &\prec^w \log y\,.&&\notag
\end{alignat}
\vspace{-5.5em}

\textrm{\\and}

\vspace{.5em}
\end{prop}
From \theref{defmaj} we can see, the mapping $\varphi\colon \R^n\to\R$ with $x\mapsto\sum_{i=1}^nf(x_i)$ where $f$ is convex has the property
\begin{equation}
x\prec y\quad\textrm{implies}\quad\varphi(x)\leq\varphi(y)\,.\notag
\end{equation}
We call this property \emph{Schur-convexity}. if $-\varphi$ is Schur-convex, so we call $\varphi$ Schur-concave. Analogously to \theref{defmaj} and \theref{defmajw} we can generalize to a important property of Schur-convex functions:

\begin{prop}\label{propSchurConvexityWithWeakMajorization}
Let $\varphi\colon D\subseteq\R^n\to\R$ be Schur-convex. If $\varphi$ is also monotone increasing (resp. monotone decreasing), then
\begin{equation}
\varphi(x)\ \leq\ \varphi(y)\qquad\textrm{for all\ \,$x,y\in D$ with $x \prec_w y$ (resp. $x\prec^w y$)}\,.\notag
\end{equation}
\end{prop}

\begin{theo}\label{theoElementarySymmetricPolynomialsAreSchurConvex_neu}
The elementary symmetric polynomials $e_k\colon\R^n\to\R$ are Schur-concave and monotone increasing for all\ \,$k\in\{1,\ldots,n\}$ and even strictly Schur-concave for $k\geq 2$.
\end{theo}

\begin{cor}\label{corMajorizationimpliesElementaryPolynomialInequalities1}
Let $x,y\in\R_+^n$. Then $x\prec y$ implies $e_k(x)\geq e_k(y)$ for all\ \,$k\in\{1,\ldots,n\}$. However $e_k(x)\geq e_k(y)$ does not imply $x\prec y$ in general: For that consider $a=(2,2,2)$ and $b=(1,1,1)$.
\end{cor}
If $x$ ist not only a permutation of $y$, the inequality between the $e_k$ is even strict for $k\geq 2$. In inverse conclusion we can say:
\begin{cor}\label{remMajProdEquImpliesPermutation}
If $x\prec y$ and $e_k(x)=e_k(y)$ for any $k\in\{2,\ldots,n\}$, then $x^\downarrow = y^\downarrow$ and thus $e_k(x)=e_k(y)$ for every $k\in\{2,\ldots,n\}$.
\end{cor}

\section{Different conditions for SSLI and the sum of powered logarithms inequality}
In the following theorems we give different conditions that guarantee the
validity of \eqref{SumOfSquaredLogarithmsInequality}. We will use Chebyshev's sum inequality in certain cases, which states:

\begin{lem}[Chebyshev's sum inequality]\label{lemChebychevsSumInequality}
If $a_1 \leq a_2 \leq \ldots \leq a_n$ and $b_1 \leq b_2 \leq \ldots \leq b_n$ are two monotone increasing sequences of real numbers, then
\begin{equation}\label{eqChebychev}
\sum_{i=1}^na_i\,b_i\ \geq\ \frac 1n\left(\sum_{i=1}^na_i\right)\left(\sum_{i=1}^nb_i\right)\ \geq\ \sum_{i=1}^na_i\,b_{n+1-i}\,.
\end{equation}
\end{lem}

\begin{theo}\label{theoQuotProdConditionForSSLI}
Let $a,b\in\R_+^n$ and there exists a rearrangement of $a$ and $b$ that satisfy
\begin{equation}
\frac{b_1}{a_1} \geq \frac{b_2}{a_2} \geq \ldots \geq \frac{b_n}{a_n}\quad\textrm{and}\quad a_1\,b_1 \geq a_2\,b_2 \geq \ldots \geq a_n\,b_n\,.
\end{equation}
If we additionally assume one of the following two conditions
\begin{equation}
e_n(a) \leq e_n(b)\quad\textrm{and}\quad e_n(a)\cdot e_n(b)\geq 1\,,\tag{3.2a}\label{beda}\\
\end{equation}
\vspace{-3.5em}

\textrm{\\or}

\vspace{-1.5em}
\begin{equation}
e_n(a) \geq e_n(b)\quad\textrm{and}\quad e_n(a)\cdot e_n(b)\leq 1\,,\tag{3.2b}\label{bedb}\addtocounter{equation}{1}
\end{equation}
then we get the sum of squared logarithms inequality \eqref{SumOfSquaredLogarithmsInequality}:
\begin{equation}
\sum_{i=1}^n(\log a_i)^2\ \leq\ \sum_{i=1}^n(\log b_i)^2\,.\notag
\end{equation}
\end{theo}
\begin{proof}
First we assume Condition \eqref{beda}:\\
Due to the monotonicity of the logarithm it follows from the assumptions that
\begin{equation}
\log\frac{b_1}{a_1} \geq \log\frac{b_2}{a_2} \geq \ldots \geq \log\frac{b_n}{a_n}\quad\textrm{and}\quad \log (a_1\,b_1) \geq \log (a_2\,b_2) \geq \ldots \geq \log (a_n\,b_n)\,.\notag
\end{equation}
Now we can estimate with Chebychev's inequality \eqref{eqChebychev} using $\tilde a_{n+k-1}\colonequals\log\frac{b_i}{a_i}$ and $\tilde b_{n+k-1}\colonequals\log a_i\,b_i:$
\begin{alignat}{2}
\sum_{i=1}^n(\log b_i)^2 - \sum_{i=1}^n(\log a_i)^2\quad&=\quad \sum_{i=1}^n \bigl((\log b_i)^2 - (\log a_i)^2\bigr)\quad=\quad \sum_{i=1}^n (\log b_i - \log a_i)(\log b_i + \log a_i)\notag\\
&=\quad \sum_{i=1}^n \;\log \frac{b_i}{a_i} \,\cdot\, \log b_ia_i\quad\overset{\eqref{eqChebychev}}{\geq}\quad \frac 1n \biggl(\sum_{i=1}^n \log\frac{b_i}{a_i}\biggr)\biggl(\sum_{i=1}^n \log b_ia_i\biggr)\notag\\
&=\quad \frac 1n \biggl(\log \underbrace{\prod_{i=1}^n \frac{b_i}{a_i}}_{e_n(b)/e_n(a)}\biggr)\biggl(\log\underbrace{\prod_{i=1}^n b_ia_i}_{e_n(a)\,e_n(b)}\biggr)\quad\geq\quad \frac 1n\log 1\log 1\quad=\quad 0\,.\notag
\end{alignat}
Finally $\sum_{i=1}^n(\log y_i)^2 - \sum_{i=1}^n(\log a_i)^2 \geq 0$ is equivalent to \eqref{SumOfSquaredLogarithmsInequality}.\\[.5em]
Next we assume condition \eqref{bedb}:\\
We set $\tilde a, \tilde b\in\R_+^n$ with $\tilde a_k \colonequals \frac 1{a_{n+1-k}}$ and $\tilde b_k \colonequals \frac 1{b_{n+1-k}}$ for $k\in\{1,\ldots,n\}$, so that we have $(\tilde a_1,\tilde a_2,\ldots, \tilde a_n) = (a_n^{-1},a_{n-1}^{-1},\ldots,a_1^{-1})$ and $(\tilde b_1,\tilde b_2,\ldots, \tilde b_n) = (b_n^{-1},b_{n-1}^{-1},\ldots,b_1^{-1})$. Then
\begin{equation}
\frac{\tilde b_k}{\tilde a_k} = \frac{\ \frac 1{b_{n+1-k}}\ }{\ \frac 1{a_{n+1-k}}\ } = \frac{a_{n+1-k}}{b_{n+1-k}}\quad\textrm{and}\quad \tilde a_k\tilde b_k = \frac 1{a_{n+1-k}b_{n+1-k}}\,.\notag
\end{equation}
Hence
\begin{equation}
\frac{\tilde b_1}{\tilde a_1} = \frac{a_n}{b_n} \geq \frac{\tilde b_2}{\tilde a_2} = \frac{a_{n-1}}{b_{n-1}} \geq \ldots \geq \frac{\tilde b_n}{\tilde a_n}\quad\textrm{and}\quad \tilde a_1\tilde b_1=\frac 1{a_nb_n} \geq \tilde a_2\tilde b_2 = \frac 1{a_{n-1}b_{n-1}} \geq\ldots\geq \tilde a_n\tilde b_n\,.\notag
\end{equation}
Furthermore
\begin{equation}
e_n(\tilde b) = \frac 1{e_n(b)} \geq \frac 1{e_n(a)} = e_n(\tilde a)\quad\textrm{and}\quad e_n(\tilde a)e_n(\tilde b) = \frac 1{e_n(a)}\cdot\frac 1{e_n(b)} \geq 1\,,\notag
\end{equation}
thus $\tilde a$ and $\tilde b$ satisfy condition \eqref{beda}. Therefore \eqref{SumOfSquaredLogarithmsInequality} holds for $\tilde a$ and $\tilde b$, and we find
\begin{align}
&\sum_{i=1}^n(\log \tilde a_i)^2\ \leq\ \sum_{i=1}^n(\log \tilde b_i)^2\qquad\equi\qquad\sum_{i=1}^n\biggl(\log \frac 1{a_{n+1-i}}\biggr)^2\ \leq\ \sum_{i=1}^n\biggl(\log \frac 1{b_{n+1-i}}\biggr)^2\notag\\
\equi\qquad&\sum_{i=1}^n\biggl(\log \frac 1{a_i}\biggr)^2\ \leq\ \sum_{i=1}^n\biggl(\log \frac 1{b_i}\biggr)^2\qquad\equi\qquad\sum_{i=1}^n\bigl(-\log a_i\bigr)^2\ \leq\ \sum_{i=1}^n\bigl(-\log b_i\bigr)^2\\
\equi\qquad&\sum_{i=1}^n\bigl(\log a_i\bigr)^2\ \leq\ \sum_{i=1}^n\bigl(\log b_i\bigr)^2\qedhere\notag
\end{align}
\end{proof}

\begin{exa}
Neither the conditions in \conjref{conjSumOfSquaredLogarithmsInequality} are stronger than in \theref{theoQuotProdConditionForSSLI} nor conversely:
\begin{itemize}
\item[i)] With $a=(14,2,10)$ and $b=(20,2,7)$ we have $e_k(a) \leq e_k(b)$ for all\ \,$k\in\{1,\ldots,n-1\}$ and $e_n(a)=e_n(b)$ but there is no rearrangement of $a$ and $b$ that satisfies $\frac{b_1}{a_1}\geq\frac{b_2}{a_2}\geq\frac{b_3}{a_3}$ and $a_1\,b_1\geq a_2\,b_2\geq a_3\,b_3$.
\item[ii)] With $a=(6,5,7)$ and $b=(10,8,3)$ we have $\frac{10}6\geq\frac 85\geq \frac 37$ and $6\cdot 10\geq 5\cdot 8\geq 7\cdot 3$. Because of $e_n(a)=210$ and $e_n(b)=240$ we have $e_n(a)\leq e_n(b)$ and $e_n(a)\,e_n(b)\geq 1$ but not $e_n(a)=e_n(b)$.
\item[iii)] With $a=(2,2,2)$ and $b=(4,2,1)$ we have $e_k(a) \leq e_k(b)$ for all\ \,$k\in\{1,\ldots,n-1\}$ and $e_n(a)=e_n(b)$. Moreover $\frac 42\geq \frac 22\geq 12$, $2\cdot 4\geq 2\cdot 2\geq 2\cdot 1$ and $e_n(a)\,e_n(b)\geq 1$.
\end{itemize}
\end{exa}

\begin{theo}[sum of powered logarithms inequality]\label{theoWeakMajorizationGivesPower}
Let $a,b\in\R^n$ and $p\in\R$ with $a_i>1$, $b_i>1$ and $p<0$. Assume $a \prec^w b$. Then
\begin{equation}\label{eqSumOfPoweredLogarithms}
\sum_{i=1}^n(\log a_i)^p\ \leq\ \sum_{i=1}^n(\log b_i)^p\,.
\end{equation}
\end{theo}
\begin{rem}
In order for $(\log a_i)^p$ and $(\log b_i)^p$ to be well defined for all $p\in\R$, we must assume $a_i >1$ and $b_i>1$.
\end{rem}
\begin{proof}
From $a \prec^w b$ with \propref{propImplicationsForLogMajorization} we obtain
\begin{equation}
\sum_{i=1}^k \log a^\uparrow_i\ \geq\ \sum_{i=1}^k \log b^\uparrow_i\qquad\textrm{for all\ \,$k\in\{1,\ldots,n\}$}\,.\notag
\end{equation}
Let $x,y\in\R_+^n$ with $x\colonequals\log a$, $y\colonequals\log b$ (therefore $x_i \colonequals \log a_i$ and $y_i \colonequals \log b_i$ for all\ \,$k\in\{1,\ldots,n\}$) then $x \prec^w y$. We now consider the function $g\colon\R_+\to\R$ with $g(z) = z^p$, then
\begin{equation}
g'(z)\ =\ \underbrace{p}_{<0}\cdot \underbrace{z^{p-1}}_{>0}\ <\ 0\qquad\textrm{and}\qquad g''(z) =\ \underbrace{p\,(p-1)}_{>0}\cdot \underbrace{z^{p-2}}_{>0}\ >\ 0\,.\notag
\end{equation}
For this reason $g$ is monotone decreasing and convex. With \theref{defmaj} we obtain $\sum_{i=1}^n x_i^p\ \leq\ \sum_{i=1}^n y_i^p$. Resubstitution now directly yields the statement.
\end{proof}

\begin{rem}
$a\prec b$ is a condition too weak and $a\prec b$ plus $e_n(a)=e_n(b)$ is too strong for \eqref{SumOfSquaredLogarithmsInequality}:

For $a=(3,2,2)$ and $b=(4,2,1)$ we get $a\prec b$ but $2.17 \approx \sum(\log a_i)^2 < \sum(\log b_i)^2 \approx 2.40$.

For $a=(4,4,4)$ and $b=(10,1,1)$ we get $a\prec b$ but $5.77 \approx \sum(\log a_i)^2 > \sum(\log b_i)^2 \approx 5.30$.

We have shown by \leref{corMajorizationimpliesElementaryPolynomialInequalities1} that $a\prec b$ implies $e_k(a)\geq e_k(b)$ (note the reverse inequality).

What about using $a\prec b$ and $e_n(a)=e_n(b)$ as sufficient requirements for \eqref{SumOfSquaredLogarithmsInequality}?

We can easily show $a\prec b$ and $e_n(a)=e_n(b)$ imply the logarithmic majorization $\log a\prec \log b$. Since the mapping $t\mapsto t^2$ is convex, it follows from \theref{defmaj} that the inequality $\sum_{i=1}^n (\log a_i)^2 \leq \sum_{i=1}^n(\log b_i)^2$ holds. However, for $a\prec b$, we can apply \coref{corMajorizationimpliesElementaryPolynomialInequalities1} to find $e_k(a)\geq e_k(b)$, hence \conjref{conjSumOfSquaredLogarithmsInequality} implies $\sum_{i=1}^n (\log a_i)^2 \geq \sum_{i=1}^n(\log b_i)^2$ and thus $\sum_{i=1}^n (\log a_i)^2 = \sum_{i=1}^n(\log b_i)^2$. This is not surprising: we already know (see \remref{remMajProdEquImpliesPermutation}) that
\begin{equation}
a\prec b\quad\textrm{and}\quad e_n(a)=e_n(b)\quad\Rightarrow\quad a^\downarrow=b^\downarrow\,.\notag
\end{equation}
Thus the vectors $a$ and $b\in\R_+^n$ are equal up to permutations.
\end{rem}

\section{Related inequalities}
\begin{prop}\label{propLogMajorizationImpliesExp}
Let $x,y\in\R_+^n$ and $m\in\R_+$. Assume additionally $\log x \prec_w \log y$, then
\begin{equation}
\sum_{i=1}^n\e^{m\,x_i}\ \leq\ \sum_{i=1}^n\e^{m\,y_i}.
\end{equation}
\end{prop}
\begin{proof}
We set $\varphi\colon \R^n\to R$ with $\varphi(x) = \sum_{i=1}^n\e^{m\,x_i}$. With \propref{propImplicationsForLogMajorization} we have $x \prec_w y$. Since $x\mapsto \e^{m\,x}$ is convex and monotone increasing, it follows directly from \theref{defmajw} that $\varphi(x) \leq \varphi(y)$.
\end{proof}

\begin{theo}\label{theoCardano}
If the real numbers $a,b,c,x,y,z$ satisfy
\begin{equation}
a + b + c\ =\ x + y + z\ =\ 0\notag
\end{equation}
and
\begin{equation}
a^2 + b^2 + c^2\ =\ x^2 + y^2 + z^2\ \neq\ 0\,,\notag
\end{equation}
then
\begin{equation}\label{u1}
\e^{xy} + \e^{yz} + \e^{zx}\ <\ \left(\e^{x^2} + \e^{y^2} + \e^{z^2}\right)\exp\Biggl(-3\sqrt[3]{\frac{1}4a^2b^2c^2}\Biggr)
\end{equation}
and
\begin{equation}\label{u2}
\e^{ab} + \e^{bc} + \e^{ca}\ <\ \left(\e^{a^2} + \e^{b^2} + \e^{c^2}\right)\exp\Biggl(-3\sqrt[3]{\frac{1}4x^2y^2z^2}\Biggr)\,.
\end{equation}
\end{theo}
\begin{proof}
For $\alpha,\beta,\gamma\in\R$ we obtain $(\alpha+\beta+\gamma)^2 = \alpha^2 + \beta^2 + \gamma^2 + 2\alpha\beta + 2\beta\gamma + 2\gamma\alpha$, therefore under the conditions
\begin{alignat}{2}
ab + bc + ca\quad&=\quad\frac 12\left((a+b+c)^2 - (a^2 + b^2 + c^2)\right)\notag\\
&=\quad\frac 12\left((x+y+z)^2 - (x^2 + y^2 + z^2)\right)\quad=\quad xy + yz + zx\notag
\end{alignat}
and we can set
\begin{equation}
p\ \colonequals\ ab + bc + ca\ =\ xy + yz + zx\,.\notag
\end{equation}
Furthermore
\begin{equation}
x^3 + px - xyz\quad =\quad x^3 + x^2y + xyz + x^2z - xyz\quad =\quad x^2(x+y+z)\quad =\quad x^2\cdot 0\quad =\quad 0\notag
\end{equation}
and analogously
\begin{equation}
y^3+py - xyz\ =\ 0\qquad\textrm{and}\qquad z^3 + pz - xyz\ =\ 0\,.\notag
\end{equation}
Therefore the cubic equation $X^3 + pX - xyz=0$ has exactly the three solutions $X\in\{x,y,z\}$. 

Following Cardano's method (see Cardano \cite{cardano1545ars}) the cubic equation $X^3 + pX + q=0$ has exactly three real roots, if and only if
\begin{equation}
D\colonequals \left(\frac q2\right)^2 + \left(\frac p3\right)^3 < 0\,.\notag
\end{equation}
Therefore we obtain
\begin{equation}
\left(\frac{-xyz}2\right)^2 + \left(\frac p3\right)^3\ =\ \frac{x^2y^2z^2}4+\frac{p^3}{27}\ <\ 0\,,\notag
\end{equation}
thus $p^3 < -27\frac{x^2y^2z^2}4$ and therefore
\begin{equation}\label{p1}
p\ =\ xy + yz + zx\ <\ -3\sqrt[3]{\frac{x^2y^2z^2}4}\,.
\end{equation}
Analogously we obtain from $X^3 + pX - abc=0$
\begin{equation}\label{p2}
p\ =\ ab + bc + ca\ <\ -3\sqrt[3]{\frac{a^2b^2c^2}4}\,.
\end{equation}
Now
\begin{equation}
p\quad =\quad p + x^2 - x^2\quad =\quad xy + yz + zx + x^2 - x^2\quad =\quad yz - x^2\,,\notag
\end{equation}
thus
\begin{equation}
yz = x^2 + p\notag
\end{equation}
and analogously $xz = y^2 + p$ and $xy = z^2 + p$. According to \eqref{p2} and the monotonicity of the exponential function we have
\begin{equation}
\e^{yz}\ <\ \e^{x^2}\exp\left(-3\sqrt[3]{\frac{a^2b^2c^2}4}\right)\,,\quad\e^{zx}\ <\ \e^{y^2}\exp\left(-3\sqrt[3]{\frac{a^2b^2c^2}4}\right)\,,\quad\e^{xy}\ <\ \e^{z^2}\exp\left(-3\sqrt[3]{\frac{a^2b^2c^2}4}\right)\notag
\end{equation}
and summing up we obtain \eqref{u1}. The proof of \eqref{u2} proceeds analogously.
\end{proof}

\begin{theo}\label{theoOrderedFunktionAndDerivateImpliesSumOverExpFunctionIsMonotone}
Let $I\subseteq \R$ and assume $f_1,\ldots,f_n\colon I\to \R$ with
\begin{equation}
\sum_{i=1}^nf_i(t)\ =\ 0\quad\textrm{and}\quad f_1(t)\ \leq\ f_2(t)\ \leq\ \ldots\ \leq f_n(t)\qquad\textrm{for all\ \,$t\in I$}\notag
\end{equation}
 and $g\colon I\to\R$ with
\begin{equation}
g(t)\ =\ \sum_{i=1}^n\e^{f_i(t)}\qquad\textrm{for all\ \,$t\in I$}\,.\notag
\end{equation}
\begin{itemize}
\item[i)] If $f'_1(t) \leq f'_2(t) \leq \ldots \leq f'_n(t)$ for all\ \,$t\in I$, then $g'(t)\geq 0$ for all\ \,$t\in I$, respectively $g$ is monotone increasing. 
\item[ii)] If $f'_1(t) \geq f'_2(t) \geq \ldots \geq f'_n(t)$ for all\ \,$t\in I$, then $g'(t)\leq 0$ for all\ \,$t\in I$, respectively $g$ is monotone decreasing.
\end{itemize}
\end{theo}
\begin{proof}
The condition $\sum_{i=1}^nf_i(t)=0$ for all\ \,$t\in I$\ \,implies 
\begin{equation}
\sum_{i=1}^nf_i'(t)=\ddt\left(\sum_{i=1}^nf_i(t)\right)=0\qquad\textrm{for all\ \,$t\in I$}\,.\notag
\end{equation}

Assume $f'_1(t)\ \leq\ f'_2(t)\ \leq\ \ldots\ \leq f'_n(t)$ for all\ \,$t\in I$. From the monotonicity of the exponential function and Chebychev's inequality with $a_i\colonequals\e^{f_i(t)}$, $b_i\colonequals f_i'(t)$, we obtain
\begin{equation}
g'(t)\quad =\quad \sum_{i=1}^nf'_i(t)\cdot \e^{f_i(t)}\quad\overset{\eqref{eqChebychev}}{\geq}\quad \frac 1n\biggl(\underbrace{\sum_{i=1}^nf'_i(t)}_{=0}\biggr)\biggl(\sum_{i=1}^n\e^{f_i(t)}\biggr)\quad =\quad 0\,.\notag
\end{equation}
Now assume instead $f'_1(t)\ \geq\ f'_2(t)\ \geq\ \ldots\ \geq f'_n(t)$ for all\ \,$t\in I$. From the monotonicity of the exponential function and Chebychev's inequality with $a_i\colonequals\e^{f_i(t)}$, $b_{n+k-i}\colonequals f_i'(t)$, we obtain
\begin{equation*}
g'(t)\quad =\quad \sum_{i=1}^nf'_i(t)\cdot \e^{f_i(t)}\quad\overset{\eqref{eqChebychev}}{\leq}\quad \frac 1n\biggl(\underbrace{\sum_{i=1}^nf'_i(t)}_{=0}\biggr)\biggl(\sum_{i=1}^n\e^{f_i(t)}\biggr)\quad =\quad 0\,.\qedhere
\end{equation*}
\end{proof}

\begin{exa}
Consider the functions $f_1,f_2,f_3\colon\R\to\R$ with
\begin{equation}
f_1(x) = -x^2 +1\,,\qquad f_2(x) = x-1\,,\qquad f_3(x) = x^2 - x\notag
\end{equation}
and
\begin{equation}
f_1'(x) = -2x\,,\qquad f_2'(x) = 1\,,\qquad f_3'(x) = 2x-1\,.\notag
\end{equation}
Then
\begin{equation}
f_1(x) + f_2(x) + f_3(x) = 0\quad\textrm{for all\ \,$x\in\R$}\,.\notag
\end{equation}
Additionally
\begin{gather}
f_1(x)\ \leq\ f_2(x)\ \leq\ f_3(x)\quad\textrm{and}\quad f_1'(x)\ \leq\ f_2'(x)\ \leq\ f_3'(x)\qquad\textrm{for all\ \,$x\in [1,\infty)$}\,,\notag\\
f_2(x)\ \leq\ f_3(x)\ \leq\ f_1(x)\quad\textrm{and}\quad f_1'(x)\ \geq\ f_2'(x)\ \geq\ f_3'(x)\qquad\textrm{for all\ \,$x\in \biggl[\frac 14,1\biggr]$}\,.\notag
\end{gather}
Now we define $g\colon\R\to\R$ with
\begin{equation}
g(x)\ =\ \e^{-x^2+1} + \e^{-1+x} + \e^{x^2-x}\ =\ \e^{f_1(x)} + \e^{f_2(x)} + \e^{f_3(x)}\,,\notag
\end{equation}
therefore we can conclude with \theref{theoOrderedFunktionAndDerivateImpliesSumOverExpFunctionIsMonotone}: $g$ is monotone increasing on $[1,\infty)$ and monotone decreasing on $[\frac 14,1]$.
\end{exa}

Now we generalize \theref{theoOrderedFunktionAndDerivateImpliesSumOverExpFunctionIsMonotone}
\begin{theo}\label{theoOrderedFunktionAndDerivateImpliesSumOverExpProdFunctionIsMonotone}
Let $I\subseteq \R$ and $f_1,\ldots,f_n\colon I\to \R$ with
\begin{equation}
\sum_{i=1}^nf_i(t)\, =\, 0\,,\quad f_1(t)\, \leq\, f_2(t)\, \leq\, \ldots\, \leq f_n(t)\quad\textrm{and}\quad f'_1(t)\, \leq\, f'_2(t)\, \leq\, \ldots\, \leq f'_n(t)\quad\textrm{for all\ \,$t\in I$}\notag
\end{equation}
and assume $g,h\colon I\to\R$ with $h$ positive and monotone increasing and
\begin{equation}
g_h(t)\ =\ \sum_{i=1}^n\e^{h(t)f_i(t)}\qquad\textrm{for all\ \,$t\in I$}\,.\notag
\end{equation}
Then $g_h'(t) \geq 0$ for all\ \,$t\in I$, which implies that $g_h$ is monotone increasing.
\end{theo}
\begin{proof}
The condition $\sum_{i=1}^nf_i(t)=0$ for all\ \,$t\in I$ implies 
\begin{equation}
\sum_{i=1}^nf_i'(t)=\ddt\left(\sum_{i=1}^nf_i(t)\right)=0\qquad\textrm{for all\ \,$t\in I$}\,.\notag
\end{equation}
With Chebychev's inequality using $a_i\colonequals f_i(t)$ resp. $a_i\colonequals f_i'(t)$ and $b_i\colonequals \e^{h(t)f_i(t)}$ we conclude
\begin{align}
g_h'(t)\quad &=\quad \sum_{i=1}^n\left(h'(t)f_i(t) + h(t)f'(t)\right)\e^{h(t)f_i(t)}\notag\\
&=\quad h'(t)\biggl(\sum_{i=1}^n f_i(t)\cdot\ \e^{h(t)f_i(t)}\biggr) + h(t)\biggl(\sum_{i=1}^nf'_i(t)\cdot\ \e^{h(t)f_i(t)}\biggr)\\
&\overset{\mathclap{\eqref{eqChebychev}}}{\geq}\quad h'(t)\cdot\frac 1n\biggl(\underbrace{\sum_{i=1}^nf_i(t)}_{=0}\biggr)\biggl(\sum_{i=1}^n\e^{h(t)f_i(t)}\biggr) + h(t)\biggl(\underbrace{\sum_{i=1}^nf_i'(t)}_{=0}\biggr)\biggl(\sum_{i=1}^n\e^{h(t)f_i(t)}\biggr)\quad=\quad 0\,.\qedhere\notag
\end{align}
\end{proof}

\begin{theo}\label{theoOrderedFunktionAndDerivateImpliesSumOverExpCompositionFunctionIsMonotone}
Let $I\subseteq \R$ and assume $f_1,\ldots,f_n\colon I\to \R$ with the properties
\begin{equation}
\sum_{i=1}^nf_i(t)\, =\, 0\,,\quad f_1(t)\,\leq\, f_2(t)\,\leq\,\ldots\,\leq f_n(t)\quad\textrm{and}\quad f'_1(t)\,\leq\ f'_2(t)\,\leq\,\ldots\,\leq f'_n(t)\notag
\end{equation}
for all\ \,$t\in I$. Additionally $h\colon D\to\R$ is given monotone increasing and convex. In this regard $D\in\R$ provides $h(f_i(t))$ is well defined for all\ \,$i\in\{1,\ldots,n\}$ and all $t\in I$. We define furthermore $H\colon I\to\R$ with
\begin{equation}
H(t)\ =\ \sum_{i=1}^n\e^{h(f_i(t))}\,.\notag
\end{equation}
Then $H'(t)\geq 0$ for all\ \,$t\in I$ and $H$ is monotone increasing.
\end{theo}
\begin{proof}
The condition $\sum_{i=1}^nf_i(t)=0$ for all\ \,$t\in I$ implies 
\begin{equation}
\sum_{i=1}^nf_i'(t)=\ddt\left(\sum_{i=1}^nf_i(t)\right)=0\qquad\textrm{for all\ \,$t\in I$}\,.\notag
\end{equation}
Because $h$ is monotone increasing $h'(x)\geq 0$ and $h(x)\leq h(y)$ for $x\leq y$, thus
\begin{equation}
h'(f_i(t))\ \geq\ 0\quad\textrm{and}\quad h(f_1(t))\ \leq\ h(f_2(t))\ \leq\ \ldots\ \leq\ h(f_n(t))\notag
\end{equation}
for all\ \,$t\in I$. By assumption $h$ is convex and $h'$ is monotone increasing, thus $h'(x)\leq h'(y)$ for $x\leq y$, therefore
\begin{equation}
h'(f_1(t))\ \leq\ h'(f_2(t))\ \leq\ \ldots\ \leq\ h'(f_n(t))\notag
\end{equation}
for all\ \,$t\in I$. If for real numbers $a_1,a_2,b_1,b_2$ the inequalities $0<a_1\leq a_2$ and $b_1\leq b_2$ are satisfied, then $a_1\,b_1 \leq a_2\,b_2$. This applied iterated, we obtain
\begin{equation}
h'(f_1(t))\cdot f_1'(t)\ \leq\ h'(f_2(t)\cdot f_2'(t)\ \leq\ \ldots\ \leq\ h'(f_n(t)\cdot f_n'(t)\notag
\end{equation}
for all\ \,$t\in I$. Finally we can easily show the original statement by applying Chebychev's inequality once with $a_i\colonequals h'(f_i(t)f'_i(t)$ and $b_i\colonequals \e^{h(f_i(t))}$, twice with $a_i\colonequals h'(f_i(t))$ and $b_i\colonequals f_i'(t)$. We obtain
\begin{align}
H'(t)\quad&=\quad \sum_{i=1}^n(h\circ f_i)'(t)\,\e^{h(f_i(t))}\quad=\quad \sum_{i=1}^n h'(f_i(t))\,f_i'(t)\ \cdot\ \e^{h(f_i(t))}\notag\\
&\overset{\mathclap{\eqref{eqChebychev}}}{\geq}\quad \frac 1n\biggl(\sum_{i=1}^nh'(f_i(t))\ \cdot f_i'(t)\biggr)\biggl(\sum_{i=1}^n\e^{h(f_i(t))}\biggr)\\
&\overset{\mathclap{\eqref{eqChebychev}}}{\geq}\quad \frac 1{n^2}\biggl(\sum_{i=1}^nh'(f_i(t))\biggr)\biggl(\underbrace{\sum_{i=1}^nf_i'(t)}_{=0}\biggr)\biggl(\sum_{i=1}^n\e^{h(f_i(t))}\biggr)\quad=\quad 0\,.\qedhere\notag
\end{align}
\end{proof}

The following theorem was proved by B\^{\i}rsan, Neff and Lankeit in \cite{Neff_log_inequality13}. Using \theref{theoOrderedFunktionAndDerivateImpliesSumOverExpFunctionIsMonotone} (respectively the generalizations \theref{theoOrderedFunktionAndDerivateImpliesSumOverExpProdFunctionIsMonotone} or \theref{theoOrderedFunktionAndDerivateImpliesSumOverExpCompositionFunctionIsMonotone}) we can now show an alternative and otherwise very elementary proof:

\begin{theo}\label{theoSumOverExponentialsIsMonotone}
Let $a,b,c,x,y,z \in\R$ with
\begin{equation}\label{monexbed1}
a \geq b \geq c\quad\textrm{and}\quad x \geq y \geq z\,.
\end{equation}
Furthermore\vspace{-1mm}
\begin{equation}\label{monexbed2}
a+b+c\ =\ x + y + z\ =\ 0\quad\textrm{and}\quad a^2 + b^2 + c^2\ =\ x^2 + y^2 + z^2\,.
\end{equation}
Then\vspace{-1mm}
\begin{equation}\label{monex}
\e^{a} + \e^{b} + \e^{c}\ \leq\ \e^{x} + \e^{y} + \e^{z}
\end{equation}
if and only if $a \leq x$.
\end{theo}
\begin{rem}
Using \theref{theoOrderedFunktionAndDerivateImpliesSumOverExpProdFunctionIsMonotone} or \theref{theoOrderedFunktionAndDerivateImpliesSumOverExpCompositionFunctionIsMonotone} instead of \theref{theoOrderedFunktionAndDerivateImpliesSumOverExpFunctionIsMonotone} the following proof even allows to show the stronger statement unter the conditions of \theref{theoSumOverExponentialsIsMonotone}:
\begin{equation}
\e^{m\,a} + \e^{m\,b} + \e^{m\,c}\ \leq\ \e^{m\,x} + \e^{m\,y} + \e^{m\,z}\qquad\textrm{if and only if}\qquad a\leq x\quad\quad\quad\textrm{for all $m\in\R_+$}\,.\notag
\end{equation}
\end{rem}
\begin{proof}
Let us first fix $a,b,c$. Then we define for simplification $r\in\R_+$ with $r^2\colonequals a^2+b^2+c^2$. From the conditions \eqref{monexbed1} and \eqref{monexbed2} we may uniquely determine $y$ and $z$ depending on $x\in\R_+$ (With \eqref{monexbed1} $x<0$ implies $a + b + c < 0 + y + z\leq 0 + 0 + 0$; a contradiction to \eqref{monexbed2}). Since $z=-x-y$ we find $y^2 + (-x-y)^2 + x^2 - r^2=0$. Let $x$ and $r$ be given, then we obtain a quadratic equation and we can solve with the quadratic formula to obtain\vspace{-3mm}
\begin{alignat}{2}
&y^2 + (-x-y)^2 + x^2 - r^2\ =\ 2y^2 + 2xy + 2x^2-r^2 = 0\quad\equi\quad y^2+xy + x^2 - \frac 12r^2\ =\ 0\notag\\
\equi\quad&y\ =\ -\frac 12 x\pm \sqrt{\frac 14x^2-x^2+\frac 12r^2}\ =\ -\frac 12x\pm \sqrt{-\frac 34x^2+\frac 12r^2}\,.\notag
\end{alignat}
Inserting these two solutions into $z=-x-y$, we get
\begin{equation}
z\ =\ -\frac 12x\mp \sqrt{-\frac 34x^2+\frac 12r^2}\,.\notag
\end{equation}
From these two possibilities $\pm$ for $y$, only the positive case, $\mp$ for $z$ only the negative case remains to satisfy \eqref{monexbed1}. Both equations have three real solutions, if and only if $-\frac 34x^2+\frac 12r^2 \geq 0$. We must have $x \leq \sqrt{\frac23r^2}$ for this. Moreover
\begin{alignat}{2}
x\ \geq\ y\quad\equi\quad&x\ \geq\ -\frac 12x + \sqrt{-\frac 34x^2+\frac 12r^2}\quad\equi\quad\frac 32x\ \geq\ \sqrt{-\frac 34x^2+\frac 12r^2}\notag\\
\equi\quad&\frac 94x^2\ \geq\ -\frac 34x^2+\frac 12r^2\quad\equi\quad3x^2\ \geq\ \frac 12r^2\quad\equi\quad x\ \geq\ \sqrt{\frac 16r^2}\,.\notag
\end{alignat}
With $D_r\colonequals  \left[\sqrt{\frac 16r^2},\sqrt{\frac23r^2}\right]$ we obtain the differentiable functions $y,z\colon D_r\to\R$ and
\begin{equation}
y(x)\ =\ -\frac 12x + \sqrt{-\frac 34x^2+\frac 12r^2}\quad\textrm{and}\quad z(x)\ =\ -\frac 12x - \sqrt{-\frac 34x^2+\frac 12r^2}
\end{equation}
as the unique solutions $(x,y(x),z(x))$ of \eqref{monexbed1} and \eqref{monexbed2}. Because the given $a,b,c$ satisfy these conditions, obviously $a\in D_r$, $b=y(a)$ and $c=y(b)$.\\[.5em]
Now we define for $m\in\R_+$ the function $\tilde g\colon D_r\to\R$ with $\tilde g(x) = \e^{m\,y(x)}+\e^{m\,z(x)}$. We know $y(x)\geq z(x)$ and 
\begin{equation}
y'(x)\ =\ \frac 12 - \frac{\frac 34x}{ \sqrt{-\frac 34x^2+\frac 12r^2}}\ \leq\ \frac 12 + \frac{\frac 34x}{ \sqrt{-\frac 34x^2+\frac 12r^2}}\ =\ z'(x)\,.\notag
\end{equation}
Thus, we can conclude with \theref{theoOrderedFunktionAndDerivateImpliesSumOverExpProdFunctionIsMonotone} or \theref{theoOrderedFunktionAndDerivateImpliesSumOverExpCompositionFunctionIsMonotone} (in both cases we can set $h(t)\colonequals m\,t$. For $m=1$ we can use directly \theref{theoOrderedFunktionAndDerivateImpliesSumOverExpFunctionIsMonotone}) that $\tilde g$ is monotone increasing. Additionally $x \mapsto \e^{m\,x}$ is monotone increasing, so building the sum $g\colon\R\to\R$ with $g(x) = \e^{m\,x} + \e^{m\,y(x)} + \e^{m\,z(x)}$ is also monotone increasing. Therefore\vspace{-1mm}
\begin{equation}
g(x)\geq g(a)\quad\textrm{if and only if}\quad x\geq a\,,\notag
\end{equation}\vspace{-1mm}
which is equivalent to the statement if we set $m=1$.
\end{proof}

\section{New Logarithmic inequalities in information theory}
First we introduce Jensen's inequality (cf. Mitrinovi{\'c}, Pe{\v{c}}ari{\'c}, \cite[eq. (2.1) p.191]{mitrinovic1993classical}) and the log sum inequality (cf. Cover, Thomas \cite[2.7 p.29]{cover2012elements}). Then we prove the information inequality which is in different notation known under the name Gibbs' inequality. The information inequality (cf. Cover, Thomas \cite[2.6 p.28]{cover2012elements}) is the most fundamental inequality in information theory. It asserts that the relative entropy between two probability distributions $p,q\colon \Omega\to [0,1]$, which is defined by
\begin{equation}\label{eq24}
D(p\parallel q)\ \colonequals\ \sum\limits_{x\in\Omega}p(x)\log \frac{p(x)}{q(x)}
\end{equation}
(or, if $p$ and $q$ are probability measures on a finite set $\Omega =\{1,\ldots,n\}$, by $\smash{D(p\parallel q)\ \colonequals \ \sum\limits_{i=1}^np_i\log \frac{p_i}{q_i}}$), is nonnegative.

\begin{lem}[Jensen's inequality]\label{lemJensenInequality}
Let $I\subseteq \R$ be an interval, $f\colon I\to\R$ be a convex function, $\lambda_1,\ldots,\lambda_n$ positive numbers with $\lambda_1+\ldots+\lambda_n = 1$ and $x_1,\ldots,x_n\in I$. Then
\begin{equation}\label{eqJensen}
f\biggl(\sum_{i=1}^n\lambda_i\,x_i\biggr)\ \leq\ \sum_{i=1}^n\lambda_i\,f(x_i)\,.
\end{equation}
Note: With $n=2$ we have directly the definition of convexity of $f$.
\end{lem}

With Jensen's inequality we can prove the so called log sum inequality:
\begin{lem}[stronger log sum inequality]\label{lemStrongerLogSumInequality}
Let $a_1,\ldots,a_n,b_1,\ldots,b_n\in\R_+$ and $k\geq 0$, then
\begin{equation}\label{eqStrongerLogSumInequality}
\sum_{i=1}^na_i\log\left(\frac{a_i}{b_i}+k\right)\ \geq\ \left(\sum_{i=1}^na_i\right)\log\left(\frac{1}{\sum_{i=1}^nb_i}\,\sum_{i=1}^na_i +k \right)\,.
\end{equation}
We have equality if and only if $\frac{a_1}{b_1} = \frac{a_2}{b_2} = \ldots = \frac{a_n}{b_n}$.
\end{lem}
\begin{rem}
For $k=0$ we get the log sum inequality:
\begin{equation}\label{eqLogSumInequality}
\sum_{i=1}^na_i\log\frac{a_i}{b_i}\ \geq\ \left(\sum_{i=1}^na_i\right)\log\left(\frac{1}{\sum_{i=1}^nb_i}\,\sum_{i=1}^na_i \right)\,.
\end{equation}
\end{rem}
\begin{proof}
With $f(x)=x\,\log (x+k)$, $\lambda_i\colonequals b_i/\sum_{i=1}^nb_i$, $x_i\colonequals a_i/b_i$ and \eqref{eqJensen} we get
\begin{align}
&\sum_{i=1}^na_i\log\biggl(\frac{a_i}{b_i}+k\biggr)\quad=\quad \biggl(\sum_{i=1}^nb_i\biggr)\,\frac{\sum_{i=1}^nb_i\frac{a_i}{b_i}\log\!\left(\frac{a_i}{b_i}+k\right)}{\sum_{i=1}^nb_i}\quad=\quad\sum_{i=1}^n\lambda_i\,x_i\,\log(x_i+k)\notag\\
=\quad &\sum_{i=1}^n\lambda_i\,f(x_i)\quad\overset{\mathclap{\eqref{eqJensen}}}{\geq}\quad f\biggl(\sum_{i=1}^n\lambda_i\,x_i\biggr)\quad =\quad \biggl(\sum_{i=1}^nb_i\biggr)\biggl(\sum_{i=1}^n\lambda_i\,x_i\biggr)\,\log\biggl(\sum_{i=1}^n\lambda_i\,x_i + k\biggr)\notag\\
=\quad &\biggl(\sum_{i=1}^nb_i\biggr)\biggl(\frac{\sum_{i=1}^na_i}{\sum_{i=1}^nb_i}\,\log\biggl(\frac{\sum_{i=1}^na_i}{\sum_{i=1}^nb_i} + k\biggr)\quad =\quad \biggl(\sum_{i=1}^na_i\biggr)\,\log\biggl(\frac{\sum_{i=1}^na_i}{\sum_{i=1}^nb_i} + k\biggr)\,.
\end{align}
If $c\colonequals \frac{a_1}{b_1} = \frac{a_2}{b_2} = \ldots = \frac{a_n}{b_n}$, we have $x_i=c$ and we get
\begin{equation}
f\biggl(\sum_{i=1}^n\lambda_i\,x_i\biggr)\ =\ f\biggl(\sum_{i=1}^n\lambda_i\,c\biggr)\ =\ f(c)\ =\ \sum_{i=1}^n\lambda_if(c)\ =\ \sum_{i=1}^n\lambda_i\,f(x_i)\,.\notag
\end{equation}
If there is an $i\in\{1,\ldots,n-1\}$ with $\frac{a_i}{b_i}\neq\frac{a_{i+1}}{b_{i+1}}$ then we get
\begin{align}
\sum_{i=1}^n\lambda_i\,f(x_i)\ &>\ f\biggl(\sum_{i=1}^n\lambda_i\,x_i\biggr)\,.\notag\qedhere
\end{align}
\end{proof}

\begin{prop}[Gibbs' inequality / Information inequality]
Let $P_n$ denote the set of probability measures on an $n$-element set, that is $P_n = \{ p\in\R_+^n\ |\ \sum_{i=1}^np_i=1\}$. The following four expressions (the first three are named \emph{Gibbs' inequality}, the last \emph{Information inequality}) are equivalent and hold for all $a,b\in P_n$:
\begin{alignat}{4}
&\phantom{ii}i)\ \ \sup_{\xi\in P_n} \biggl\{ \prod_{i=1}^n\xi_i^{a_i}\biggr\} &&= \prod_{i=1}^na_i^{a_i}\,,\quad&&ii)\ \inf_{\xi\in P_n} \biggl\{ \sum_{i=1}^na_i\,(-\log \xi_i)\biggr\} = \sum_{i=1}^na_i\,(-\log a_i)\,,\notag\\
&iii)\ \sum_{i=1}^na_i\,(-\log b_i) &&\geq \sum_{i=1}^na_i\,(-\log a_i)\,,\quad\quad\quad\quad&& iv)\ D(a\parallel b)=\sum_{i=1}^na_i\log\frac{a_i}{b_i} \geq 0\,.
\end{alignat}
\end{prop}
\begin{proof}
First we prove the equality of the four expressions:
\begin{alignat}{2}
&\sup_{b\in P_n} \biggl\{ \prod_{i=1}^nb_i^{a_i}\biggr\}\ =\ \prod_{i=1}^na_i^{a_i}\quad\equi\quad \inf_{b\in P_n} \biggl\{ -\prod_{i=1}^nb_i^{a_i}\biggr\}\ =\ -\prod_{i=1}^na_i^{a_i}\quad\equi\quad -\prod_{i=1}^nb_i^{a_i}\ \geq\ -\prod_{i=1}^na_i^{a_i}\notag\\
\equi\quad &-\log \prod_{i=1}^nb_i^{a_i}\ \geq\ -\log \prod_{i=1}^na_i^{a_i}\quad\equi\quad -\sum_{i=1}^n \log b_i^{a_i}\ \geq\ -\sum_{i=1}^n\log a_i^{a_i}\notag\\
\equi\quad &\sum_{i=1}^n a_i(-\log b_i)\ \geq\ \sum_{i=1}^na_i(-\log a_i)\quad\equi\quad\sum_{i=1}^n a_i(\log a_i - \log b_i)\ \geq\ 0\notag\\
\equi\quad &\sum_{i=1}^n a_i\log\frac{a_i}{b_i}\ \geq\ 0\quad\equi\quad D(a\parallel b)\ \geq\ 0\,.\notag
\end{alignat}
Now let $a,b\in P_n$. With the stronger log sum inequality and because $\sum_{i=1}^na_i = \sum_{i=1}^nb_i = 1$ we get
\begin{equation}\label{eqStrongerInformationInequality}
\sum_{i=1}^n a_i\log\left(\frac{a_i}{b_i} + k\right)\ \geq\ \left(\sum_{i=1}^na_i\right)\log\left(\frac 1{\sum_{i=1}^nb_i}\sum_{i=1}^na_i + k\right)\ =\ \log(1+k)\,.
\end{equation}
With $k=0$ we obain inequality case iv) (information inequality).
\end{proof}
\begin{rem}
Analogously we can denote inequality \eqref{eqStrongerInformationInequality} as the \emph{stronger information inequality}.
\end{rem}
\begin{cor}[generalized log sum inequality]\label{lemGeneralizedLogSumInequality}
Let $a,b\in\R_+^n$ and $k_1,\ldots,k_n\in [0,\infty)$. Then
\begin{equation}\label{eqGeneralizedLogSumInequality}
\sum_{i=1}^na_i\log\prod_{s=1}^m\left(\frac{a_i}{b_i}+k_s\right)\ \geq\ \biggl(\sum_{i=1}^na_i\biggr)\log\prod_{s=1}^m\left(\frac{1}{\sum_{i=1}^nb_i}\,\sum_{i=1}^na_i+k_s\right)\,.
\end{equation}
\end{cor}
\begin{proof}
With $m$ numbers $k_1,\ldots,k_m\in [0,1)$ we obtain by $m$ times building sums of the stronger log sum inequality \eqref{eqStrongerLogSumInequality}
\begin{equation}
\sum_{s=1}^m\left(\sum_{i=1}^na_i\log\left(\frac{a_i}{b_i}+k_s\right)\right)\ \geq\ \sum_{s=1}^m\left(\sum_{i=1}^na_i\log\!\left(\frac{1}{\sum_{i=1}^nb_i}\,\sum_{i=1}^na_i+k_s\right)\right)\,.\notag
\end{equation}
With use of distributivity and laws of logarithms we directly obtain the desired result.
\end{proof}

\begin{cor}[generalized information inequality]
Let $a,b\in P_n$ and $k_1,\ldots,k_n \in [0,\infty)$. Then
\begin{equation}\label{eqGeneralizedInformationInequality}
\sum_{i=1}^na_i\,\log\prod_{s=1}^m\left(\frac{a_i}{b_i} + k_s\right)\ \geq\ \log\prod_{s=1}^m(1+k_s)\,.
\end{equation}
\end{cor}
\begin{proof}
We simplify \eqref{eqGeneralizedLogSumInequality} under the condition $a,b\in P_n$, therefore $\sum_{i=1}^na_i = \sum_{i=1}b_i=1$. \end{proof}


\end{document}